\documentclass[a4paper,11pt,draft]{article}
\usepackage{amsmath,amssymb,amsfonts,latexsym, amsthm, mathrsfs}
\usepackage{chngcntr}

\usepackage[active]{srcltx}
\catcode`\@=11
\@addtoreset{equation}{section}

\catcode`\@=12

\newtheorem{Theorem}{Theorem}[section]
\newtheorem{Lemma}[Theorem]{Lemma}
\newtheorem{Proposition}[Theorem]{Proposition}
\newtheorem{Corollary}[Theorem]{Corollary}
\newtheorem{Remark}[Theorem]{Remark}
\newtheorem{Definition}{Definition}[section]

\newcommand{\Rn}{{\mathbb R}^{N}}
\newcommand{\N}{\mathbb{N}}

\newcommand{\al} {\alpha}

\newcommand{\ga} {\gamma}

\newcommand{\De} {\Delta}
\newcommand{\la} {\lambda}

\newcommand{\pa}{\partial}
\newcommand{\na} {\nabla}

\newcommand{\bdw}{\partial\Omega}

\newcommand{\R}{\mathbb{R}}

\newcommand{\deb}{\rightharpoonup}

\newcommand{\Iom}{\int_{\Omega}}

\begin{document}

\title{Semilinear elliptic PDE's with biharmonic operator \\ and a singular potential}

\author{Mousomi Bhakta\\
 Department of Mathematics,\\ Indian Institute of Science Education and Research,\\
Dr. Homi Bhaba road,  Pune-411008, India.\\
Email: {\it mousomi@iiserpune.ac.in}}

\date{}

\maketitle

\begin{abstract}
\noindent
\footnotesize We study the existence/nonexistence of positive solution to the problem of the type:
\begin{equation}\tag{$P_{\la}$}
\begin{cases}
\Delta^2u-\mu a(x)u=f(u)+\la b(x)\quad\textrm{in $\Omega$,}\\
u>0 \quad\textrm{in $\Omega$,}\\
u=0=\Delta u \quad\textrm{on $\bdw$,}
\end{cases}
\end{equation}
where $\Omega$ is a smooth bounded domain in $\Rn$, $N\geq 5$, $a, b, f$ are nonnegaive functions satisfying certain hypothesis which we will specify later. $\mu,\la$ are positive constants. 
Under some suitable conditions on functions $a, b, f$ and the constant $\mu$, we show that there exists $\la^*>0$ such that when $0<\la<\la^*$, ($P_{\la}$) admits a  solution in $W^{2,2}(\Omega)\cap W^{1,2}_0(\Omega)$ and for $\la>\la^*$, it does not have any solution in $W^{2,2}(\Omega)\cap W^{1,2}_0(\Omega)$. Moreover as $\la\uparrow\la^*$, minimal positive solution of ($P_{\la}$) converges in $W^{2,2}(\Omega)\cap W^{1,2}_0(\Omega)$ to a solution of ($P_{\la^*}$).  We also prove that there exists $\tilde{\la}^*<\infty$ such that $\la^*\leq\tilde{\la}^*$ and  for $\la>\tilde{\la}^*$, the above problem ($P_{\la}$) does not have any solution even in the distributional sense/very weak sense and there is complete {\it blow-up}. Under an additional integrability condition on $b$, we establish the uniqueness of positive solution of ($P_{\la^*}$) in $W^{2,2}(\Omega)\cap W^{1,2}_0(\Omega)$.

\bigskip


\end{abstract}

\section{Introduction}

In this article  we  study  the semilinear fourth order elliptic problem with singular potential:
\begin{equation}\tag{$P_{\la}$}
\begin{cases}
\Delta^2u-\mu a(x)u=f(u)+\la b(x)\quad\textrm{in $\Omega$,}\\
u>0 \quad\textrm{in $\Omega$,}\\
u=0=\Delta u \quad\textrm{on $\bdw$,}
\end{cases}
\end{equation}
where
$\Delta^2 u=\Delta(\Delta u)$, $\Omega$ is a smooth bounded domain in $\Rn$, $N\geq 5$. $a, b, f$ are nonnegative functions. $a\in L^1_{loc}(\Omega)$, $ b\in L^2(\Omega)$, $b\not\equiv 0$. $\mu, \la$ are (small) positive constants.   
We assume that
\begin{equation}\label{assum-f}
f:\R^{+}\to\R^{+} \quad\text{is a convex} \quad C^1 \quad\text{function with}\quad f(0)=0=f'(0) 
\end{equation}
and satisfying the following growth conditions:
\begin{equation}\label{assum-f1}
\lim_{t\to\infty}\frac{f(t)}{t}=\infty,
\end{equation}
\begin{equation}\label{assum-f2}
\int_{1}^{\infty}g(s)ds<\infty \quad\text{and}\quad sg(s)<1 \quad\text{for}\quad s>1, 
\end{equation}
where we define, for $s\geq 1$,
\begin{equation}\label{def-g}
g(s)=\sup_{t>0}\frac{f(t)}{f(ts)}.
\end{equation}
It is easy to see that $g$ is nonincreasing, nonnegative function. Since by convexity $t\to\frac{f(t)}{t}$ is increasing and $f(0)=0$, it follows that $s\to sg(s)$ is nonincreasing.

From literature we know that the usual norm in $W^{k,p}(\Omega)$ is $\displaystyle\left(\int_{\Omega}\sum_{0\leq|\alpha|\leq k}|D^{\alpha}u|^p dx\right)^\frac{1}{p}$. Thanks to interpolation theory, one can neglect intermediate derivatives and see that 
\begin{equation}\label{norm-1}
||u||_{W^{k,p}(\Omega)}=\displaystyle\left(\Iom|u|^p dx+\Iom|D^k u|^p dx\right)^\frac{1}{p},
\end{equation}
 defines a norm which is equivalent to the usual norm in $W^{k,p}(\Omega)$ (see \cite{A}). As $\Omega$ is a smooth bounded domain and $W^{k,p}_0(\Omega)$ is the closure of $C_0^{\infty}(\Omega)$ w.r.t. the norm in  $W^{k,p}(\Omega)$, invoking \cite[Theorem 2.2]{GGS} we find that
\begin{equation}\label{norm-2}
||u||_{W^{k,p}_0(\Omega)}=\displaystyle\left(\Iom| D^k u|^p dx\right)^\frac{1}{p},
\end{equation}
defines an equivalent norm to \eqref{norm-1}. Now onwards we will consider $W^{k,p}_0(\Omega)$ endowed with the norm defined in \eqref{norm-2}. The inner product in $W^{2,2}(\Omega)\cap W^{1,2}_0(\Omega)$ is defined by
$$(u, v)_{W^{2,2}(\Omega)\cap W^{1,2}_0(\Omega)}=\Iom\Delta u\Delta v dx,$$ which induces the norm 
\begin{equation}\label{norm-3}
||u||_{W^{2,2}(\Omega)\cap W^{1,2}_0(\Omega)}=|\Delta u|_{L^2(\Omega)},
\end{equation}
equivalent to \eqref{norm-2} with $k=p=2$ (for details see \cite{GGS}, \cite{PP}).

We assume $a\in L^1_{loc}(\Omega)$ satisfies the following condition: there exists a positive constant $\ga>0$ such that
\begin{equation}\label{assum-a'}
\Iom \displaystyle\left(|\Delta u|^2-a(x)^2u^2\right)dx\geq \ga\Iom u^2 \quad\forall\quad u\in C^{\infty}_0(\Omega).
\end{equation}
Using Fatou's lemma and the standard density argument, it is easy to check that  \eqref{assum-a'} holds for every $u\in W^{2,2}\cap W^{1,2}_0(\Omega)$. Therefore we write
\begin{equation}\label{assum-a}
\Iom \displaystyle\left(|\Delta u|^2-a(x)^2u^2\right)dx\geq \ga\Iom u^2 \quad\forall\quad u\in W^{2,2}\cap W^{1,2}_0(\Omega).
\end{equation}
In addition, 
\begin{equation}\label{assum-mu}
0<\mu<\sqrt{\ga}.
\end{equation}
Using \eqref{assum-a'} and \eqref{assum-mu} it follows that
\begin{equation}\label{mu-ga}
\mu\Iom a(x)u^2 dx\leq\mu\displaystyle\left(\Iom a(x)^2u^2 dx\right)^\frac{1}{2}\left(\Iom u^2 dx\right)^\frac{1}{2}\leq\frac{\mu}{\sqrt{\ga}}|\Delta u|^2_{L^2(\Omega)} \quad\forall\quad u\in C^{\infty}_0(\Omega).
\end{equation}
Therefore,
$$||u||_{H}^2 :=\Iom\displaystyle\left[|\Delta u|^2-\mu a(x)u^2\right]dx, $$ is a norm in $C^{\infty}_0(\Omega)$ and completion of $C^{\infty}_0(\Omega)$ with respect to this norm yields the Hilbert space $H$. By \eqref{mu-ga}, \eqref{assum-mu} and \eqref{norm-2}, it follows that $||u||_H$ is equivalent to $||u||_{W^{2,2}_0(\Omega)}$. Thanks to \eqref{mu-ga},  the norm equivalence established above and the Poincare inequality, there exists $\tilde{\ga}>0$ such that 
\begin{equation}\label{tilde-ga}
\Iom\displaystyle\left(|\Delta u|^2-\mu a(x)u^2\right)dx\geq \tilde{\ga}\Iom u^2 dx \quad\forall \   \  u\in \mathcal C^{\infty}_{0}(\Omega).
\end{equation}
\eqref{tilde-ga} implies first eigenvalue of $\Delta^2-\mu a(x)$ is strictly positive.

\vspace{2mm}

We  note that if $a(x)=\frac{\al}{|x|^2}$ where $\al<\bar\al:= \frac{N(N-4)}{4}$, applying the following Rellich inequality (\cite{Rel54}, \cite{Rel69}):
\begin{equation}\label{Rellich}
\int_{\Rn}|\De u|^2dx\geq \displaystyle\bar\al^2\int_{\Rn}|x|^{-4}|u|^2 dx\quad\forall \   \  u\in \mathcal C^{\infty}_{0}(\R^N),
\end{equation}
and the Poincare inequality along with the norm equivalence established above, it is not diffcicult to check that  \eqref{assum-a} holds. When $a(x)=\frac{\bar\al}{|x|^2}$, \eqref{assum-a} is the improved Hardy-Rellich inequality (see \cite{GGM}, \cite{TZ}). However, if $\al>\bar\al$, and $0\in\Omega$, \eqref{assum-a} fails to hold and there is no $u\not\equiv 0$ such that $\Delta^2u-a(x)^2u\geq 0$,

\begin{Definition}\label{d: sol}
We say that $u\in W^{2,2}(\Omega)\cap W^{1,2}_0(\Omega)$ is a solution of $(P_{\la})$ if $u>0$ a.e. ,
$f(u)\in L^2(\Omega)$ and $u$ satisfies the following: 
$$\Iom \displaystyle\left( \Delta u \Delta\phi-\mu a(x)u\phi\right)dx=\Iom\left(f(u)+\la b(x)\right)\phi\ dx\quad\forall\quad \phi\in W^{2,2}(\Omega)\cap W^{1,2}_0(\Omega).$$
Similarly $u\in W^{2,2}(\Omega)\cap W^{1,2}_0(\Omega)$ is called a supersolution (\it subsolution) if $f(u)\in L^2(\Omega)$ and for all positive $\phi\in W^{2,2}(\Omega)\cap W^{1,2}_0(\Omega)$,
$$\Iom \displaystyle\left( \Delta u \Delta\phi-\mu a(x)u\phi\right)dx\geq (\leq)\Iom\left(f(u)+\la b(x)\right)\phi\ dx.$$
\end{Definition}

\begin{Definition}\label{d:weak sol}
We say that $u\in L^1(\Omega)$ is a distributional solution or very weak solution of $(P_{\la})$ if $u>0$ a.e. , \\
$\mu a(x)u+f(u)\in L^1_{loc}(\Omega)$ and $u$ satisfies ($P_{\la}$) in the distributional sense, i.e., 
$$\Iom \displaystyle u\left(\Delta^2 \phi-\mu a(x)\phi\right) dx=\Iom\left(f(u)+\la b(x)\right)\phi\ dx \quad\forall\quad \phi\in C^{\infty}_{0}(\Omega).$$
\end{Definition}

Similar type of problem  with the Laplace operator  in much more generalized sense was extensively studied by Dupaigne and Nedev in \cite{DN}. In \cite{DN}, the authors have proved an necessary and sufficient condition for the existence of $L^1$ solution and they have also established an estimate from above and below for the solution. We also refer \cite{BDT}, \cite{BV}, \cite{D} (and the references therein) for the related problems in the second order case.

Higher order problem is quite different compared to the second order case. In this case a possible failure of the maximum principle causes
several technical difficulties. Possibly because of this reason the knowledge on higher order nonlinear problems is far from being reasonably complete, as it is in the second order case. In the case of fourth order problem Navier boundary conditions play an important role to prove existence
results as under this boundary condition, equation with biLaplacian operator can be rewritten as a second order system with Dirichlet boundary value problems. Then using classical elliptic theory,  one can easily prove a Maximum Principle.  As a consequence,  one can deduce a Comparison Principle which plays as one of the key factor in proving existence results. In a recent work \cite{PP}, an equation similar to ($P_{\la}$) with $a(x)=\frac{1}{|x|^4}$ and $f(u)=u^p$ has been studied. More precisely, in \cite{PP} the authors have studied the optimal power p for existence/nonexistence of distributional solutions. In recent years there are many papers dealing with $W^{2,2}(\Omega)\cap W^{1,2}_0(\Omega)$ solution of semilinear elliptic and parabolic problem with biLaplacian operator and some specific nonlinearities. We quote a few among them \cite{AGGM}, \cite{BG}, \cite{DDGM}, \cite{EGP} (also see the references therein). Semilinear elliptic equations with biharmonic operator arise in continuum mechanics, bio- physics, differential geometry. In particular in the modeling of thin elastic plates, clamped
plates and in the study of the Paneitz-Branson equation and the Willmore equation (see \cite{GGS} and the references therein for more details).

The paper is organized as follows:

In Section 2 we recall some useful lemmas from \cite{PP} and prove some important lemmas regarding existence. In Section 3 we prove our main existence result.  More precisely, under some hypothesis on $f$, we prove there exists $\la^{*}>0$ such that if $0<\la<\la^{*}$, problem ($P_{\la}$) has a minimal solution $u_{\la}$ in $W^{2,2}(\Omega)\cap W^{1,2}_0(\Omega)$. Moreover, if $\la>\la^{*}$, then ($P_{\la}$) does not have any solution which belongs to $W^{2,2}(\Omega)\cap W^{1,2}_0(\Omega)$. Under an additional mild growth condition on $f$ at infinity, we also prove when $\la\uparrow \la^*$, there exists $u^*\in W^{2,2}(\Omega)\cap W^{1,2}_0(\Omega)$ such that minimal solution $u_{\la}$ of $(P_{\la})$ converges to $u^*$ in $W^{2,2}(\Omega)\cap W^{1,2}_0(\Omega)$ and $u^*$ happens to be a solution of $(P_{\la^*})$. Section 4  deals with the case for which ($P_{\la}$) does not have any solution even in the very weak sense. In this case we establish {\it complete blow-up} phenomenon (see Definition \ref{d:blow-up}). Section 5 is devoted to the stability result where the minimal positive solution in $W^{2,2}(\Omega)\cap W^{1,2}_0(\Omega)$ already exists. In this section, under some better integrability condition on $b$, we also prove $(P_{\la^*})$ has a unique solution in $W^{2,2}(\Omega)\cap W^{1,2}_0(\Omega)$.

\section{Some important Lemmas}
 
\begin{Definition}
We say that $u\in L^1(\Omega)$ is a weak supersolution (subsolution) to
$$\Delta^2 u=g(x,u) \quad\text{in}\quad\Omega,$$
in the sense of distribution if $g(x,u)\in L^1(\Omega)$ and for all positive $\phi\in C^{\infty}_{0}(\Omega)$, we have
$$\Iom u\Delta^2 \phi dx\geq(\leq)\Iom g(x, u)\phi dx.$$
If $u$ is a weak supersolution and as well a weak subsolution in the sense of distribution, then we say that $u$ is a distributional solution.
\end{Definition}

Next we recall three important lemmas from \cite{PP} which we will use frequently in this paper.
\begin{Lemma}\label{SMP}({\textit Strong Maximum Principle}).
Let $u$ be a nontrivial supersolution to
\begin{equation}\label{eq:smp}
\begin{cases}
\Delta^2 u=0 \quad\text{in}\quad\Omega,\\
u=0=\Delta u \quad\text{on}\quad\bdw.
\end{cases}
\end{equation}
Then $-\Delta u>0$ and $u>0$ in $\Omega$.
\end{Lemma}
\begin{proof}
See \cite[Lemma 3.2]{PP}.
\hfill$\square$
\end{proof}
\begin{Lemma}\label{comp prin}({\textit Comparison Principle}).
Let $u$ and $v$ satisfy the following prob:
\begin{equation}\label{eq:cp}
\begin{cases}
\Delta^2 u\geq \Delta^2 v \quad\text{in}\quad\Omega,\\
u\geq v \quad\text{on}\quad\bdw,\\
-\Delta u\geq -\Delta v \quad\text{on}\quad\bdw.
\end{cases}
\end{equation}
Then, $-\Delta u\geq -\Delta v$ and $u\geq v$ in $\Omega$.
\end{Lemma}
\begin{proof}
See \cite[Lemma 3.3]{PP}.
\hfill$\square$
\end{proof}
\begin{Lemma}\label{WHP}({\textit Weak Harnack Principle})\cite[Lemma 3.4]{PP}.\\
Let $u$ be a positive distributional supersolution to  \eqref{eq:smp}. Then for any $B_R(x_0)\Subset\Omega$, there exists a positive constant $C=C(\theta, \rho, q, R)$, $0<q<\frac{N}{N-2}$, $0<\theta<\rho<1$, such that
$$||u||_{L^q(B_{\rho R}(x_0))}\leq C \text{ess} \inf_{B_{\theta R}(x_0)}u.$$
\end{Lemma}
\begin{Lemma}\label{l:ex}
Let $a\in L^1_{loc}(\Omega)$, $b\in L^2(\Omega)$ , $a, b\geq 0$ a.e., $b\not\equiv 0$, $\mu$ is a positive constant satisfying \eqref{assum-mu} and $a$ satisfies \eqref{assum-a}.  Then the equation
\begin{equation}\label{eq-1}
\begin{cases}
\Delta^2 u-\mu a(x)u=b \quad\text{in}\quad\Omega,\\
u=0=\Delta u \quad\text{on}\quad\bdw,
\end{cases}
\end{equation}
has a positive solution $u\in W^{2,2}(\Omega)\cap W^{1,2}_0(\Omega)$. 
\end{Lemma}
\begin{proof}
Given $b\in L^2(\Omega)$, we know there exists unique $u_1\in W^{2,2}\cap W^{1,2}_0(\Omega)$ satisfying the following prob:
\begin{equation*}
\begin{cases}
\Delta^2 u_1=b \quad\text{in}\quad\Omega,\\
u_1=0=\Delta u_1 \quad\text{on}\quad\bdw.
\end{cases}
\end{equation*}
Applying strong maximum principle ( Lemma \ref{SMP}) we get $u_1>0$.
Now define $u_n$ ($n\geq 2$) as follows:
\begin{equation}\label{u-n}
\begin{cases}
\Delta^2 u_n=\mu a(x)u_{n-1}+b \quad\text{in}\quad\Omega,\\
u_n=0=\Delta u_n \quad\text{on}\quad\bdw.
\end{cases}
\end{equation}
By \eqref{assum-a}, we have $\mu a(x)u_{n-1}\in L^2(\Omega)$. This in turn implies the existence of unique $u_n\in W^{2,2}\cap W^{1,2}_0(\Omega)$ which satisfies \eqref{u-n}. Also by comparison principle we have $0<u_1\leq\cdots\leq u_{n-1}\leq u_n\leq\cdots$.  \\
{\bf Claim:} $\{u_n\}$ is a Cauchy sequence in $W^{2,2}\cap W^{1,2}_0(\Omega)$.\\
To see this, we note that $\Delta^2(u_{n+1}-u_n)=\mu a(x)(u_{n}-u_{n-1})$. By taking $(u_{n+1}-u_n)$ as a test function and using \eqref{assum-a}, we get
\begin{eqnarray*}
|\Delta(u_{n+1}-u_n)|_{L^2(\Omega)}^2 &=& \mu\Iom a(x)(u_{n}-u_{n-1})(u_{n+1}-u_n) dx\\
&\leq& \mu \displaystyle\left(\Iom a(x)^2(u_{n}-u_{n-1})^2dx\right)^\frac{1}{2}\left(\Iom(u_{n+1}-u_{n})^2dx\right)^\frac{1}{2}\\
&\leq&\frac{\mu}{\sqrt{\ga}}|\Delta(u_{n}-u_{n-1})|_{L^2(\Omega)}|\Delta(u_{n+1}-u_n)|_{L^2(\Omega)}.
\end{eqnarray*}
Therefore $|\Delta(u_{n+1}-u_n)|_{L^2(\Omega)}\leq\frac{\mu}{\sqrt{\ga}}|\Delta(u_{n}-u_{n-1})|_{L^2(\Omega)}\leq\cdots\leq (\frac{\mu}{\sqrt{\ga}})^{n-1}|\Delta(u_2-u_1)|_{L^2(\Omega)} $.
As $\mu <\sqrt{\ga}$, from the above estimate we can conclude that $\{u_n\}$ is a Cauchy sequence in $W^{2,2}(\Omega)\cap W^{1,2}_0(\Omega)$.
Hence, there exists $u\in W^{2,2}(\Omega)\cap W^{1,2}_0(\Omega)$ such that $u_n\to u$ in $W^{2,2}(\Omega)\cap W^{1,2}_0(\Omega)$. Moreover, $u>0$ since $u_n>0 \quad\forall\quad n\geq 1$. As $u_n\in W^{2,2}(\Omega)\cap W^{1,2}_0(\Omega)$ solves \eqref{u-n}, we have
$$\Iom \Delta u_n\Delta \phi dx=\mu\Iom a(x)u_{n-1}\phi dx+\Iom b\phi dx\quad\forall\quad\phi \in W^{2,2}(\Omega)\cap W^{1,2}_0(\Omega).$$
Taking the limit as $n\to\infty$, we obtain $u$ is a solution to \eqref{eq-1}. 
\hfill$\square$
\end{proof}

\vspace{2mm}

\begin{Lemma}\label{l:supersol}
Let $a\in L^1_{loc}(\Omega)$, $b\in L^2(\Omega)$,  $f:\R^{+}\to\R^{+}$ ($f$ convex) be nonnegative functions. Let $\mu,\la>0$, $\mu<\sqrt{\ga}$.\\
Suppose there exists a nonnegative supersolution $\tilde{u}\in W^{2,2}(\Omega)\cap W^{1,2}_0(\Omega)$ of $(P_{\la})$ ( respectively for \eqref{eq-1}) . Then there exists a unique  solution $u\in W^{2,2}(\Omega)\cap W^{1,2}_0(\Omega)$ to $(P_{\la})$ which satisfies $0\leq u\leq \tilde{w}$ for any supersolution $\tilde{w}\geq 0$ of  $(P_{\la})$ (respectively for \eqref{eq-1}). $u$ is called the minimal nonnegative solution of $(P_{\la})$ (respectively for \eqref{eq-1}). By strong maximum principle it also follows that $u>0$ in $\Omega$. 
\end{Lemma}

\begin{Remark}\label{r:1}
We denote the minimal positive  solution of \eqref{eq-1} by $\zeta_1$ and denote $G(b)=\zeta_1$.
The function $0<u\in W^{2,2}(\Omega)\cap W^{1,2}_0(\Omega)$ solving ($P_{\la}$) (respectively \eqref{eq-1}) also solves ($P_{\la}$) (\eqref{eq-1}) in the distributional sense (see definition \eqref{d:weak sol}).  
\end{Remark}

\begin{proof}
The proof is same for both the equations $(P_{\la})$ and \eqref{eq-1}, therefore we present here the proof for $(P_{\la})$. First we will show that if minimal solution exists then it is unique. To see this, let $u_1$ and $u_2$ are two solutions which satisfy $0\leq u_i\leq\tilde{w}, ( i=1,2)$ for every nonnegative supersolution $\tilde{w}$. Thus $u_1\leq u_2$ and $u_2\leq u_1$. Hence $u_1=u_2$.

Next, let $\tilde{u}\geq 0$ be a  supersolution to ($P_{\la}$) and $u_0\in W^{2,2}(\Omega)\cap W^{1,2}_0(\Omega)$ be a positive solution to 
\begin{equation*}
\begin{cases}
\Delta^2 u_0=\la b \quad\text{in}\quad\Omega,\\
u_0=0=\Delta u_0 \quad\text{on}\quad\bdw.
\end{cases}
\end{equation*} 
By comparison principle we get $0<u_0\leq\tilde{u}$ in $\Omega$. Next, using iteration we will show that there exists $u_n\in W^{2,2}(\Omega)\cap W^{1,2}_0(\Omega)$ for $n=1,2, \cdots$ such that $u_n$ solves the following problem:
\begin{equation}\label{eq:un-1}
\begin{cases}
\Delta^2 u_n=\mu a(x)u_{n-1}+f(u_{n-1})+\la b(x) \quad\text{in}\quad\Omega,\\
u_n=0=\Delta u_n \quad\text{on}\quad\bdw.
\end{cases}
\end{equation} 
Since $\tilde{u}$ is a weak supersolution to ($P_{\la}$), we have $f(\tilde{u})\in L^2(\Omega)$. Thanks to the fact that $0<u_0\leq\tilde{u}$ and $f$ is convex (thus $f$ is nondecreasing), we get $f(u_0)\leq f(\tilde{u})$. Thus $f(u_0)+\la b(x)\in L^2(\Omega)$. Also, by \eqref{assum-a} it follows that $\mu a(x)u_{0}\in L^2(\Omega)$. Therefore $u_1$ is well defined and by comparison principle $0<u_0\leq u_1\leq \tilde{u}$. Using the induction method, similarly we can show that $u_n$ is well defined and $0<u_0\leq u_1\leq\cdots\leq u_n\leq\cdots\leq \tilde{u}$. 

\vspace{2mm}

{\bf Claim:} $\{u_n\}$ is uniformly bounded in $W^{2,2}(\Omega)\cap W^{1,2}_0(\Omega)$.\\
To see this, let us note that from \eqref{eq:un-1} we can write
\begin{eqnarray*}
|\Delta u_n|^2_{L^2(\Omega)} &=&\Iom\displaystyle\left(\mu a(x)u_{n-1}+f(u_{n-1})+\la b(x)\right)u_n dx \\
&\leq&\Iom \displaystyle\left(\mu a(x)\tilde{u}^2+f(\tilde{u})\tilde{u}+\la b\tilde{u}\right)dx\\
&\leq& \big[\mu |a(x)\tilde{u}|_{L^2(\Omega)}+|f(\tilde{u})|_{L^2(\Omega)}+\la |b|_{L^2(\Omega)}\big]|\tilde{u}|_{L^2(\Omega)}\\
&\leq& C.
\end{eqnarray*}
As a consequence there exists  $u\in W^{2,2}(\Omega)\cap W^{1,2}_0(\Omega)$ such that upto a subsequence $u_n\deb u$ in $W^{2,2}(\Omega)\cap W^{1,2}_0(\Omega)$ and $u_n\to u$ in $L^2(\Omega)$. From \eqref{eq:un-1} we have,
$$\Iom\Delta u_n\Delta\phi dx=\Iom\displaystyle\left[\mu a(x)u_{n-1}+f(u_{n-1})+\la b\right]\phi dx \quad\forall\quad\phi\in W^{2,2}(\Omega)\cap W^{1,2}_0(\Omega).$$
Using Vitaly's convergence theorem we can pass to the limit $n\to\infty$ on the RHS and obtain $u$ is a solution to $(P_{\la})$. Also $u>0$ since $u_n>0$ for all $n\geq 1$.

Let $\tilde{w}$ be another supersolution, then by comparison principle it follows that $u_0\leq \tilde{w}$ and $u_{n}\leq\tilde{w}$ for every $n\geq 1$. Taking the limit $n\to\infty$, it gives us that $u\leq\tilde{w}$. Hence the lemma follows.
\hfill$\square$
\end{proof}

\section{Existence and nonexistence results}
\begin{Theorem}\label{main ex}
Assume $a\in L^1_{loc}(\Omega)$, $0\not\equiv b\in L^2(\Omega)$, $a,b,f$ are nonnegative functions,  \eqref{assum-a}, \eqref{assum-mu}, \eqref{assum-f}, \eqref{assum-f1}, \eqref{assum-f2} and \eqref{def-g} are satisfied. Let $G=(\Delta^2-\mu a(x))^{-1}$ and $\zeta_1=G(b)$, as proved in Lemma \ref{l:ex} (also see Remark \ref{r:1}).
Suppose there exists constants $\epsilon>0$ and $C>0$ such that
\begin{equation}\label{ex-cond}
f(\epsilon\zeta_1)\in L^2(\Omega) \quad\text{and}\quad G(f(\epsilon\zeta_1))\leq C\zeta_1 \quad a.e.
\end{equation}
Then there exists $0<\la^{*}=\la^{*}(N, a(x), b(x), f, \mu)$  such that

  if $\la<\la^{*}$, then ($P_{\la}$) has a minimal positive solution $u_{\la}\in W^{2,2}(\Omega)\cap W^{1,2}_0(\Omega)$ and $u_{\la}\geq\la\zeta_1$.
  
  If $\la>\la^{*}$ then $(P_{\la})$ has no positive solution in $W^{2,2}(\Omega)\cap W^{1,2}_0(\Omega)$.
  
Moreover, if $\la>0$ is {\it small} then  $$\la\zeta_1\leq u_{\la}\leq 2\la\zeta_1.$$

\end{Theorem}

To prove this theorem, first we need to prove a lemma and a  proposition.

\begin{Lemma}\label{l:ex1}
Let the functions $a, b$ and the constant $\mu$ satisfy the assumptions in Theorem \ref{main ex}. $\zeta_1=G(b)$ as in theorem \ref{main ex} and assume that \eqref{assum-f} is satisfied.   If 
$$f(2\zeta_1)\in L^2(\Omega)\quad\text{and}\quad G(f(2\zeta_1))\leq \zeta_1,$$ then $(P_1)$ admits a solution  $u\in W^{2,2}(\Omega)\cap W^{1,2}_0(\Omega)$.
\end{Lemma}
\begin{proof}
Let $f(2\zeta_1)\in L^2(\Omega)$ and $G(f(2\zeta_1))\leq \zeta_1$. We define, $v :=G(f(2\zeta_1))+\zeta_1$. Clearly $v>0$ and $v\in W^{2,2}(\Omega)\cap W^{1,2}_0(\Omega)$ since $\zeta_1$ and $G(f(2\zeta_1))$ are in $W^{2,2}(\Omega)\cap W^{1,2}_0(\Omega) $ by Lemma \ref{l:ex}. Also, $$v-\zeta_1=G(f(2\zeta_1)), \quad  v\leq 2\zeta_1 \quad\text{and}\quad f(v)\in L^2(\Omega).$$ Thus we have, 
$$\Delta^2(v-\zeta_1)-\mu a(x)(v-\zeta_1)=f(2\zeta_1)\quad \text{in}\quad\Omega,$$ i.e.,  
$$\Delta^2 v-\mu a(x)v=f(2\zeta_1)+b\geq f(v)+b \quad\text{in}\quad \Omega$$ and $v=0=\Delta v$  on $\bdw$. As a result, $v$ is a positive supersolution of $(P_1)$. Finally, by applying Lemma \ref{l:supersol} we get the existence of minimal positive solution  $u\in W^{2,2}(\Omega)\cap W^{1,2}_0(\Omega)$ of $(P_1)$.
\hfill$\square$
\end{proof}

\begin{Proposition}\label{p:ex}
Suppose there exists $\tilde{\la}>0$ such that $(P_{\tilde{\la}})$ has a positive solution $u_{\tilde{\la}}\in W^{2,2}(\Omega)\cap W^{1,2}_0(\Omega)$. Then for every $0<\la<\tilde{\la}$, $(P_{\la})$ has a solution in $W^{2,2}(\Omega)\cap W^{1,2}_0(\Omega)$.
\end{Proposition}
\begin{proof}
Let $u_{\tilde{\la}}\in W^{2,2}(\Omega)\cap W^{1,2}_0(\Omega)$ denote a positive solution corresponding to $(P_{\tilde{\la}})$. Therefore by definition (see Definition \ref{d: sol}) $f(u_{\tilde{\la}})\in L^2(\Omega)$. Define, $v=\tilde{\la}\zeta_1$. Note that,
$$\Delta^2\displaystyle\left(\frac{u_{\tilde{\la}}}{\tilde{\la}}\right)-\mu a(x)\left(\frac{u_{\tilde{\la}}}{\tilde{\la}}\right)= \frac{1}{\tilde{\la}}\displaystyle\left(f(u_{\tilde{\la}})+\tilde{\la}b \right)=\frac{f(u_{\tilde{\la}})}{\tilde{\la}}+b\geq b \quad\text{in}\quad\Omega.$$
This implies, $\frac{u_{\tilde{\la}}}{\tilde{\la}}$ is a positive supersolution to \eqref{eq-1}. Therefore by minimality of $\zeta_1$ it follows, $\zeta_1\leq \frac{u_{\tilde{\la}}}{\tilde{\la}}$, which in turn implies  $v\leq u_{\tilde{\la}}$. 
Let $0<\la<\tilde{\la}$ and define, $w=u_{\tilde{\la}}-v+\la\zeta_1$. Clearly $w>0$. Using the definition of $v$ and $\la$ we also get $w\leq u_{\tilde{\la}}$. By convexity of $f$, it follows $\frac{f(t)}{t}$ is increasing and thus $f$ is nondecreasing. As a consequence, $f(w)\leq f(u_{\tilde{\la}})$ and hence $f(w)\in L^2(\Omega)$. Also,
$$\Delta^2 w-\mu a(x)w=f(u_{\tilde{\la}})+\tilde{\la}b-(\tilde{\la}-\la)b=f(u_{\tilde{\la}})+\la b\geq f(w)+\la b.$$
As a result, $w\in W^{2,2}(\Omega)\cap W^{1,2}_0(\Omega)$ is a positive supersolution to $(P_{\la})$. Hence by Lemma \ref{l:supersol}, there exists minimal positive solution of $(P_{\la})$.
\hfill$\square$
\end{proof}

\vspace{2mm}

{\bf Proof of Theorem \ref{main ex}:} We assume \eqref{ex-cond} holds true. \\
{\bf Step 1}: In this step we will show that if $\la>0$ is small then $(P_{\la})$ has a positive a solution $u_{\la}\in W^{2,2}(\Omega)\cap W^{1,2}_0(\Omega)$. We will prove this step in the spirit of \cite{DN}.
By Lemma \ref{l:ex1}, it follows that $(P_{\la})$ has a solution as soon as it holds, 
\begin{equation}\label{ex-cond-1}
f(2\la\zeta_1)\in L^2(\Omega)  \quad\text{and}\quad G(f(2\la\zeta_1))\leq\la\zeta_1.
\end{equation}
From the definition of $g$ (see definition \eqref{def-g}), it follows that $g(\frac{\epsilon}{2\la})\geq \frac{f(t)}{f(t\frac{\epsilon}{2\la})}$ for all $t>0$. Choosing $t=2\la\zeta_1$, we get $f(2\la\zeta_1)\leq f(\epsilon\zeta_1)g(\frac{\epsilon}{2\la})$. Applying \eqref{ex-cond}, we have $f(2\la\zeta_1)\in L^2(\Omega)$ and $G(f(2\la\zeta_1))$ is well defined. Also by minimality of $G(f(2\la\zeta_1))$ and by assumption \eqref{ex-cond}, we get
$$G(f(2\la\zeta_1))\leq g\displaystyle\left(\frac{\epsilon}{2\la}\right)G(f(\epsilon\zeta_1))\leq Cg\left(\frac{\epsilon}{2\la}\right)\zeta_1.$$
To show \eqref{ex-cond-1} holds for $\la>0$ small, it is enough to prove that 
$$\lim_{\la\to 0}\frac{1}{\la}g\displaystyle\left(\frac{\epsilon}{2\la}\right)=0 \quad\text{or equivalently}\quad \lim_{K\to\infty}Kg(K)=0.$$
Since $s\to sg(s)$ is nonincreasing, the above limit is well defined, i.e. there exists $C'\geq0$ such that $\lim_{K\to\infty}Kg(K)=C'$. If $C'>0$, then $g(K)\sim\frac{C}{K}$ near $\infty$ and this contradicts \eqref{assum-f2}. Hence $C'=0$ and \eqref{ex-cond-1} holds for $\la>0$ small.

{\bf Step 2}: Define, $$\Lambda=\{\la>0:  (P_{\la}) \ \text{has a minimal positive solution}\ u_{\la} \}, $$ 
By Step 1 and Proposition \ref{p:ex}, it follows that $\Lambda$ is a non-empty interval. We define,
$$\la^{*}=\sup \Lambda .$$
Then it is easy to see that, if $\la<\la^{*}$, $(P_{\la})$ has a minimal positive solution and for $\la>\la^{*}$, $(P_{\la})$ does not have any positive solution in $W^{2,2}(\Omega)\cap W^{1,2}_0(\Omega)$.

{\bf Step 3}: From $G(b)=\zeta_1$, it is easy to see that $G(\la b)=\la\zeta_1$. If $\la<\la^{*}$ and $u_{\la}$ denotes the corresponding minimal positive solution of $(P_{\la})$, then it is not difficult to check that $u_{\la}$ is a supersolution to the equation satisfied by $\la\zeta_1$.  Therefore by minimality of 
$\la\zeta_1$, we get 
\begin{equation}\label{u-low}
u_{\la}\geq\la\zeta_1.
\end{equation}

{\bf Step 4}: In this step we will show that if $\la>0$ is small, then $$\la\zeta_1\leq u_{\la}\leq 2\la\zeta_1.$$
By Step 1, \eqref{ex-cond-1} holds since $\la>0$ is small. Define, $w=G(f(2\la\zeta_1))+\la\zeta_1$. Therefore 
$$w\leq 2\la\zeta_1 \quad\text{and}\quad w-\la\zeta_1=G(f(2\la\zeta_1)).$$
As in the proof of Lemma \ref{l:ex1}, we can establish that $w\in W^{2,2}(\Omega)\cap W^{1,2}_0(\Omega)$ is a positive supersolution of $(P_{\la})$. Thus $u_{\la}\leq w\leq 2\la\zeta_1$. Combining this with \eqref{u-low}, we have  $\la\zeta_1\leq u_{\la}\leq 2\la\zeta_1.$

\hfill$\square$

\vspace{2mm}

Define 
\begin{equation}\label{u-star}
u^*(x)=\lim_{\la\uparrow\la^*}u_{\la}(x), \quad x\in\Omega.
\end{equation}

\begin{Theorem}
Assume all the assumptions in Theorem \ref{main ex} are satisfied and $u_{\la}$ denotes the minimal positive solution of $(P_{\la})$ for $0<\la<\la^*$. In addition we suppose $f$ satisfies the following condition:
\begin{equation}\label{assum-f3}
\lim_{s\to\infty}\frac{sf'(s)}{f(s)}>1.
\end{equation}
Then $u^{*}\in W^{2,2}(\Omega)\cap W^{1,2}_0(\Omega)$ and $u^*$ is a solution to $(P_{\la^*})$. Moreover, $u_{\la}\to u^*$ in  $W^{2,2}(\Omega)\cap W^{1,2}_0(\Omega)$.
\end{Theorem}
\begin{proof}
$u_{\la}$ is a solution of $(P_{\la})$ implies 
\begin{equation}\label{eq:u-la}
\Iom\Delta u_{\la}\Delta v=\mu\Iom a(x)u_{\la}v+\Iom f(u_{\la})v+\la\Iom b(x)v \quad \forall\quad v\in W^{2,2}(\Omega)\cap W^{1,2}_0(\Omega).
\end{equation}
By Theorem \ref{t:stability}, it follows that $u_{\la}$ is a stable solution of $(P_{\la})$ (see Definition \ref{d:stable}). Therefore $\displaystyle\Iom\left(|\Delta u_{\la}|^2-\mu a(x)u_{\la}^2-f'(u_{\la})u_{\la}^2\right)dx\geq 0$. Hence by taking $v=u_{\la}$ in \eqref{eq:u-la} we have,
\begin{equation}\label{eq:u-la1}
\displaystyle\Iom f'(u_{\la})u_{\la}^2 dx\leq\Iom\left(|\Delta u_{\la}|^2-\mu a(x)u_{\la}^2\right)dx=\Iom\left(f(u_{\la})u_{\la}+\la b(x)u_{\la}\right)dx.
\end{equation}
Moreover, using \eqref{assum-f3} we can write, for every $\epsilon>0$ there exists $C>0$ such that 
\begin{equation}\label{eq:f}
(1+\epsilon)f(s)s\leq f'(s)s^2+C \quad\forall\quad s\geq 0.
\end{equation}
Hence combining \eqref{eq:u-la1} and \eqref{eq:f} we obtain,
$$(1+\epsilon)\Iom \displaystyle\left(f'(u_{\la})u_{\la}^2-\la b(x)u_{\la}\right)dx\leq (1+\epsilon)\Iom f(u_{\la})u_{\la} dx\leq\Iom \left(f'(u_{\la})u_{\la}^2+C\right)dx.$$
As a result, $$\epsilon\Iom f'(u_{\la})u_{\la}^2 dx\leq C|\Omega|+(1+\epsilon)\la\Iom b u_{\la}dx.$$
Consequently, 
\begin{equation}\label{eq:u-la2}
\Iom f(u_{\la})u_{\la} dx\leq C_1+C_2\la\Iom b u_{\la}dx,
\end{equation}  for some constants $C_1, C_2>0$.
Since $\la<\la^*$,  by taking $v=u_{\la}$ in \eqref{eq:u-la} and applying Holder inequality and \eqref{eq:u-la2} we have,
\begin{eqnarray*}
\Iom |\Delta u_{\la}|^2 dx &=&\mu\Iom a(x)u^2_{\la}+\Iom f(u_{\la})u_{\la}+\la\Iom bu_{\la}\\
&\leq& \mu |a(x)u_{\la}|_{L^2(\Omega)}|u_{\la}|_{L^2(\Omega)}+\la^{*}(1+C_2)\Iom bu_{\la}dx+C_1.
\end{eqnarray*}
Applying \eqref{assum-a} and Cauchy-Schwartz inequality with $\delta>0$ on the above estimate, we get
\begin{eqnarray*}
\Iom |\Delta u_{\la}|^2 dx &\leq& \frac{\mu}{\sqrt{\ga}}|\Delta u_{\la}|^2_{L^2(\Omega)}+C_3|b|_{L^2(\Omega)}|u_{\la}|_{L^2(\Omega)}+C_1\\
&\leq& \frac{\mu}{\sqrt{\ga}}|\Delta u_{\la}|^2_{L^2(\Omega)}+\frac{C_3}{\sqrt{\ga}}|b|_{L^2(\Omega)}|\Delta u_{\la}|_{L^2(\Omega)}+C_1\\
&\leq& \frac{\mu}{\sqrt{\ga}}|\Delta u_{\la}|^2_{L^2(\Omega)}+\delta |\Delta u_{\la}|^2_{L^2(\Omega)}+c(\delta)|b|^2_{L^2(\Omega)}+C_1.
\end{eqnarray*}
Since $\mu<\sqrt{\ga}$ (by  \eqref{assum-mu}), we can choose $\delta>0$ such that 
$\frac{\mu}{\sqrt{\ga}}+\delta<1$. Hence from the above estimate we have
$$\Iom |\Delta u_{\la}|^2 dx\leq C_4|b|^2_{L^2(\Omega)}+C_1\leq C',$$ for some constant $C'>0$. This implies $\{u_{\la}\}$ is uniformly bounded in $W^{2,2}(\Omega)\cap W^{1,2}_0(\Omega)$ for $\la<\la^*$. Consequently, by \eqref{u-star} we conclude that $u_{\la}\deb u^*$ in $W^{2,2}(\Omega)\cap W^{1,2}_0(\Omega)$. Passing the limit $\la\to\la^*$ in \eqref{eq:u-la} via. Lebesgue monotone convergence theorem it is easy to check that $u^*$ is a solution to $(P_{\la^*})$.
When $\la\to\la^*$, using monotone convergence theorem we also have,
\begin{eqnarray*}
||u_{\la}||^2_{W^{2,2}(\Omega)\cap W^{1,2}_0(\Omega)}=\Iom|\Delta u_{\la}|^2dx &=&\mu\Iom a(x)u_{\la}^2+\Iom f(u_{\la})u_{\la}+\la\Iom bu_{\la}\\
&\to&\mu\Iom a(x){u^*}^2+\Iom f(u^*)u^*+\la^*\Iom bu^*\\
&=& \Iom|\Delta u^*|^2dx=||u^*||^2_{W^{2,2}(\Omega)\cap W^{1,2}_0(\Omega)}
\end{eqnarray*}
Hence $||u_{\la}||_{W^{2,2}(\Omega)\cap W^{1,2}_0(\Omega)}\to ||u^*||_{W^{2,2}(\Omega)\cap W^{1,2}_0(\Omega)}$. Combining this along with the weak convergence, we conclude $u_{\la}\to u^*$ in $W^{2,2}(\Omega)\cap W^{1,2}_0(\Omega)$.
\hfill$\square$
\end{proof}

\vspace{2mm}

{\bf Remark:} We denote by $u_{\la^{*}}$, the minimal positive solution of $(P_{\la^*})$.

\section{Nonexistence of very weak solution and Complete blow-up }
Define $$\tilde{\lambda}^{*}=\sup\{\la>0:  (P_{\la}) \ \text{has a very weak solution/distributional solution }\}.$$
It is not difficult to check that if $u\in W^{2,2}(\Omega)\cap W^{1,2}_0(\Omega)$ is a solution to $(P_{\la})$ in the sense of Definition \ref{d: sol}, then $u$ is a very weak solution of $(P_{\la})$ as well. Therefore $\tilde{\lambda}^{*}\geq\la^{*}$.
\begin{Lemma}\label{l:blow-up}
$\tilde{\lambda}^{*}<\infty$.
\end{Lemma}
\begin{proof}
Assume $(P_{\la})$ has a very weak solution $u\in L^1(\Omega)$. Therefore 
\begin{equation}\label{eq:dist}
\Iom \displaystyle u\left(\Delta^2 \phi-\mu a(x)\phi\right) dx=\Iom\left(f(u)+\la b(x)\right)\phi\ dx \quad\forall\quad \phi\in C^{\infty}_{0}(\Omega).
\end{equation}
Let $\tilde{\Omega}\Subset\Omega$ and $\psi\in C^{\infty}_0(\Omega)$ be a  nonnegative function such that supp$(\psi)\subset\tilde{\Omega}$. We choose $\phi$ as follows:
\begin{equation*}
\begin{cases}
\Delta^2 \phi=\psi \quad\text{in}\quad\Omega,\\
\phi=0=\Delta\phi \quad\text{on}\quad\bdw.
\end{cases}
\end{equation*}
Clearly $\phi\in C^{\infty}(\Omega)$ and by strong maximum principle $\phi> 0$ in $\Omega$. Thus there exists $c>0$ such that $\phi\geq c>0$ in $\tilde{\Omega}$. Substituting this $\phi$ in \eqref{eq:dist}, we have
\begin{equation}\label{eq:phi-psi}
\mu\Iom a(x)u\phi\ dx+\Iom f(u)\phi\ dx+\la\Iom b(x)\phi\ dx=\Iom u\psi\ dx=\int_{\tilde{\Omega}}u\psi\ dx.
\end{equation}
Since $f$ satisfies \eqref{assum-f1}, it is easy to check that, for $\epsilon>0$ there exists a constant $C_{\epsilon}>0$ such that
$$u\leq C_{\epsilon}+\epsilon f(u).$$
Therefore from RHS of \eqref{eq:phi-psi} we get,
\begin{equation*}
\int_{\tilde{\Omega}}u\psi\ dx\leq C_{\epsilon}\Iom\psi dx+\epsilon\int_{\tilde{\Omega}}f(u)\psi dx\leq  C_{\epsilon}\Iom\psi dx+\epsilon|\frac{\psi}{\phi}|_{L^{\infty}(\tilde{\Omega)}}\int_{\Omega}f(u)\phi dx
\end{equation*}
Now choose $\epsilon>0$ such that $\epsilon|\frac{\psi}{\phi}|_{L^{\infty}(\tilde{\Omega)}}<\frac{1}{2}$. Thus from \eqref{eq:phi-psi} we have,
$$\mu\Iom a(x)u\phi\ dx+\frac{1}{2}\Iom f(u)\phi\ dx+\la\Iom b(x)\phi\ dx\leq C\Iom\psi dx\leq C'.$$
This implies $\tilde{\la}^{*}<\infty$. In particular there are no solution of $(P_{\la})$ for $\la>\tilde{\la}^{*}$, even in the very weak sense.
\hfill$\square$
\end{proof}
\begin{Definition}\label{d:blow-up}
Let $\{a_n(x)\}$, $\{b_n(x)\}$ and $\{f_n\}$ be increasing sequence of bounded functions converging pointwise respectively to $a(x)$, $b(x)$ and $f$. (Since $f\in C^1(\R^+)$, without loss of generality we can also assume $f_n\in C(\R^+)$). Let $u_n\in W^{2,2}(\Omega)\cap W^{1,2}_0(\Omega)$ be the minimal nonnegative solution of 
\begin{equation}\label{a-n}
\begin{cases}
\Delta^2 u_n-\mu a_n(x)u_n=f_n(u_n)+\la b_n(x) \quad\text{in}\quad\Omega,\\
u_n=0=\Delta u_n \quad\text{on}\quad\bdw.
\end{cases}
\end{equation}
We say that there is a complete blow-up in $(P_{\la})$, if given any such $\{a_n(x)\}$, $\{b_n(x)\}$, $\{f_n\}$ and $u_n$,
$$u_n(x)\to\infty \quad\forall\ x\in\Omega.$$
\end{Definition}

{\bf Remark}: Existence of $u_n$ follows from Lemma A.3.

\vspace{2mm}

We prove the next theorem in the spirit of \cite{PP}.
\begin{Theorem}\label{t:blow-up}
Fix $\la>0$. Suppose $(P_{\la})$ does not have any solution, even in the very weak sense. Then there is complete blow up.
\end{Theorem}
\begin{proof}
Let $u_n\in W^{2,2}(\Omega)\cap W^{1,2}_0(\Omega)$ be the minimal nonnegative solution of \eqref{a-n}. Using the monotonicity property of $a_n, b_n$ and $f_n$, we get $u_{n+1}$ is a supersolution of the equation satisfied by $u_n$. Thus $u_n\leq u_{n+1}$. Therefore to establish the blow-up result, it is enough to show the complete blow-up for the family of minimal solution $u_n$.

We will prove this by method of contradiction. Assume there exists $x_0\in\Omega$ and a positive constant $C$ such that $u_n(x_0)\leq C$. Thus applying weak Harnack inequality (Lemma \ref{WHP}) we have
$$|u_n|_{L^1(B_{\rho R}(x_0))}\leq C ess\inf_{B_{\theta R}(x_0)}u_n\leq Cu_n(x_0)\leq C',$$ where $0<\theta<\rho<1$.
Then following the same argument as in \cite{PP}, we can show that there exists $r>0$ and a positive constant $C=C(r)$ such that 
$$\int_{B_r(0)}u_n dx\leq C, \quad\text{ uniformly in}\quad n\in\N.$$
Therefore, applying monotone convergence theorem we see that, there exists $u\geq 0$ such that $u_n\to u$ in $L^1(B_r(0))$.

Let $\phi$ be the solution to the problem:
\begin{equation*}
\begin{cases}
\Delta^2 \phi=\chi_{B_r(0)} \quad\text{in}\quad\Omega,\\
\phi=0=\Delta\phi \quad\text{on}\quad\bdw.
\end{cases}
\end{equation*}
Clearly $\phi\in W^{4,p}(\Omega)$ since $\chi_{B_r(0)}\in L^p(\Omega)$ for all $p\geq 1$. Taking $\phi$ as a test function in \eqref{a-n}, we have
$$\Iom\displaystyle\left(a_n(x)u_n\phi+f_n(u_n)\phi+\la b_n\phi \right)dx=\int_{B_r(0)}u_n dx\leq C.$$
By monotone convergence theorem and Fatou's lemma, it follows that
$$a_n(x)u_n\uparrow a(x)u \quad\text{in}\quad L^1_{loc}(B_r(0))$$
$$ f_n(u_n)\to f(u)\quad\text{in}\quad L^1_{loc}(B_r(0)) \quad\text{and}\quad b_n(x)\uparrow b(x) \quad\text{in}\quad L^1_{loc}(B_r(0)).$$  
Hence $u$ is a very weak solution to $(P_{\la})$ in $B_{r_1}(0)\Subset B_r(0)$, which contradicts the assumption of this theorem. 
\hfill$\square$
\end{proof}
Combining Lemma \ref{l:blow-up} and Theorem \ref{t:blow-up}, we get the following corollary.
\begin{Corollary}
If $\la>\tilde{\la}^{*}$, then there is complete blow-up.
\end{Corollary}

\section{Stability results}
\begin{Definition}\label{d:stable}
We say that $u\in W^{2,2}(\Omega)\cap W^{1,2}_0(\Omega)$ is a stable solution, if the first eigenvalue of the linearized operator of the equation $(P_{\la})$ is nonnegative, i.e., if 
$$\inf_{\phi\in C^{\infty}_0(\Omega)\setminus\{0\}}\frac{\displaystyle\Iom (|\Delta\phi|^2-\mu a(x)\phi^2-f'(u)\phi^2)dx}{\displaystyle\Iom\phi^2 dx}\geq 0.$$ 
\end{Definition}

\begin{Theorem}\label{t:stability}
Suppose all the assumptions in Theorem \ref{main ex} are satisfied and for $0<\la<\la^{*}$,   let $u_{\la}$ denote the minimal positive solution of $(P_{\la})$.  Then $u_{\la}$ is stable.
\end{Theorem}

\begin{proof}
Following the idea of Dupaigne and Nedev \cite{DN}, we prove this theorem. Let $a_n(x)=\text{min}(a(x), n)$, $b_n=\text{min}(b(x), n)$ and $u_n\in W^{2,2}(\Omega)\cap W^{1,2}_0(\Omega)$ denote the minimal positive solution of the following problem
\begin{equation}\label{a-n1}
\begin{cases}
\Delta^2 u_n-\mu a_n(x)u_n=f(u_n)+\la b_n(x) \quad\text{in}\quad\Omega,\\
u_n=0=\Delta u_n \quad\text{on}\quad\bdw.
\end{cases}
\end{equation}
By Lemma \ref{l:supersol}, $u_n$ is well defined since $u_{\la}$ is a supersolution of \eqref{a-n1}. Let $\la_1^n(\Delta^2-\mu a_n(x)-f'(u_n))$ denote the 1st eigenvalue of the linearized operator $\Delta^2-\mu a_n(x)-f'(u_n)$.

\vspace{2mm}

{\bf Claim}: $\la_1^n(\Delta^2-\mu a_n(x)-f'(u_n))\geq 0$. \\
To prove the claim, we choose $p>N$. Define, $I: \R\times W^{4,p}(\Omega)\to L^p(\Omega)$ as follows
$$I(\la, u)=\Delta^2 u-\mu a_n(x)u-f(u)-\la b_n.$$ An easy computation using \eqref{tilde-ga} and implicit function theorem, (see \cite{DN}) it follows that there exists a unique maximal curve $\la\in [0, \la^{\#})\to u(\la)$ such that
$$I(\la, u(\la))=0 \quad\text{and}\quad I_u(\la, u(\la))\in Iso(W^{4,p}. L^p).$$
If $0<\la<\la^{\#}$, then $u_n\leq u(\la)$, since $u_n$ is the minimal positive solution of \eqref{a-n1}. Thus $f(u_n)\leq f(u(\la))$. Moreover, $I(\la, u(\la))=0$ implies $f(u(\la))=\Delta^2 u(\la)-\mu a_n(x)u(\la)-\la b_n(x)\in L^p(\Omega)$, which in turn implies $f(u_n)\in L^p(\Omega)$. Therefore by elliptic regularity theory, $u_n$ is in the domain of $I$ and hence $u_n=u(\la)$.\\
Following the same method as in \cite{DN}, we can show that if $0<\la<\la^{*}$, $u_n$ is in the domain of $I$. Thus $\la^{\#}=\la^{*}$ (otherwise we could extend the curve $u(\la)$ beyond $\la^{\#}$ contradicting its maximality). We also claim that the first eigenvalue of $I_u(\la, u_n)$ does not vanish for any $\la<\la^{*}$. To see this, assume $\phi$ is an the eigen function corresponding to this first eigenvalue. If the first eigenvalue vanishes for some $\la_0<\la^{*}$, then we have $\Delta^2\phi-\mu a(x)\phi-f'(u_n)\phi=0$, i.e., $I_u(\la_0, u_n)=0$ but we know that $I_u(\la, u)$ can not vanish for any $\la<\la^{\#}$ (otherwise $u(\la)$ will not be the maximal curve). Consequently, since $\la^{\#}=\la^{*}$, we can say that  the first eigenvalue of $I_u(\la, u_n)$ does not vanish for any $\la<\la^{*}$. Moreover, by \eqref{tilde-ga} we know first eigenvalue of $I_u(0,0)$ is strictly positive. Therefore we conclude that $\la_1^n(\Delta^2-\mu a_n(x)-f'(u_n))\geq 0$ for every $\la\in[0,\la^{*})$.

Also, $\{u_n\}$ is a nondecreasing sequence and converges to a solution of $(P_{\la})$ in $W^{2,2}(\Omega)\cap W^{1,2}_0(\Omega)$. Since $u_n\leq u_{\la}$, $\lim_{n\to\infty}u_n$ has to be the minimal solution $u_{\la}$. Therefore by monotone convergence theorem we conclude the first eigenvalue $\la_1(\Delta^2-\mu a(x)-f'(u_{\la}))\geq 0$ which completes the proof.
\hfill$\square$
\end{proof}

\begin{Theorem}
Suppose all the assumptions in Theorem \ref{main ex} hold and $u_{\la}$ denote the minimal positive solution of $(P_{\la})$. We also assume \eqref{assum-f3} is satisfied. If $\la=\la^{*}$ and $b\in L^p(\Omega)$ for some $p>\frac{N}{3}$, then $u_{\la^*}$ is the only positive solution of $(P_{\la^{*}})$ which belongs to $\in W^{2,2}(\Omega)\cap W^{1,2}_0(\Omega)$.
\end{Theorem}

\begin{proof}
Suppose the theorem does not hold and $u$ and $v$ are two distinct positive solutions of $(P_{\la^{*}})$, where $u, v\in W^{2,2}(\Omega)\cap W^{1,2}_0(\Omega) $. Let $u$ be the minimal positive solution. Therefore $u\leq v$. Applying strong maximal principle we can easily check that $u<v$ in $\Omega$. Since $u$ and $v$ are solution, by Definition \eqref{d: sol} we have $f(u), f(v)\in L^2(\Omega)$. Thus applying \eqref{assum-a}, we get $\mu a(x)u+f(u)+\la^{*}b\in L^2(\Omega)$. This together with the elliptic regularity theory gives $u\in W^{4,2}(\Omega)\cap W^{1,2}_0(\Omega)$. Similarly same result holds for $v$ as well. Define $w=\frac{u+v}{2}$. Then $w\in W^{4,2}(\Omega)\cap W^{1,2}_0(\Omega)$ and by convexity of $f$, we have
$$f(w)=\displaystyle f\left(\frac{u+v}{2}\right)\leq\frac{f(u)+f(v)}{2}\in L^2(\Omega).$$ Thus,
$$\Delta^2 w-\mu a(x)w=\frac{f(u)+f(v)}{2}+\la^{*}b\geq f(w)+\la^{*}b.$$
Thus $w$ is a supersolution of $(P_{\la^{*}})$. By Lemma A.1, it follows that $w$ is a solution to $(P_{\la^{*}})$. As a consequence, inequality on the above expression becomes equality and by convexity of $f$ we conclude that $f$ is linear on $[u(x), v(x)]$ for almost every $x\in\Omega$. For $\epsilon\in(0,1)$, define $\theta=\epsilon u+(1-\epsilon)v$. Therefore $f''(\theta(x))$ exists for a.e $x\in\Omega$ and $f''(\theta(x))=0$ a.e. $x\in\Omega$. 
This implies $\na(f'(\theta))=0$ a.e. in $\Omega$, which in turn implies $f'(\theta)=C$ a.e. in $\Omega$ and $f(\theta)=C\theta+D$ a.e. in $\Omega$ for some constant $C$ and $D$. Moreover, using convexity of $f$, this implies $f(t)=Ct+D$ for $t\in[\text{ess inf}\ \theta, \text{ess sup}\ \theta]$. Applying Lemma A.2, we have ess inf $\theta=0$. Since $f(0)=0=f'(0)$, we get $f\equiv 0$ on $[0, \text{ess sup}\ \theta]$. As $\epsilon>0$  arbitrary, we can conclude $f\equiv 0$ on $[0, \text{ess sup}\ v]$. Therefore $u$ and $v$ both satisfy the following prob:
\begin{equation*}
\begin{cases}
\Delta^2 u-\mu a(x)u=\la^{*} b(x) \quad\text{in}\quad\Omega,\\
u=0=\Delta u \quad\text{on}\quad\bdw.
\end{cases}
\end{equation*}
This in turn implies, $v-u$ satisfies the following prob:
\begin{equation*}
\begin{cases}
\Delta^2(v-u)-\mu a(x)(v-u)=0 \quad\text{in}\quad\Omega,\\
v-u=0=-\Delta(v- u) \quad\text{on}\quad\bdw.
\end{cases}
\end{equation*}
This contradicts \eqref{tilde-ga} since $v-u\in W^{2,2}(\Omega)\cap W^{1,2}_0(\Omega)$. Hence $u=v$. 
\hfill$\square$
\end{proof}

\appendix
\section{Appendix}

\begin{Lemma} If $b\in L^p(\Omega)$ for some $p>\text{max} (2,\frac{N}{3})$ and $w\in W^{4,2}(\Omega)\cap W^{1,2}_0(\Omega)$ is a supersolution of $(P_{\la^{*}})$, then $w$ is a solution of $(P_{\la^{*}})$.
\end{Lemma}
\begin{proof}
Let $w$ be a supersolution of $(P_{\la^{*}})$ and not a solution.  Define, $\nu\in\mathcal{D}'(\Omega)$ by $$\nu(\phi)=\Iom w(\Delta^2\phi)-\displaystyle\left(\mu a(x)w+f(w)+\la^{*}b\right)\phi \quad\forall\ \phi\in C^{\infty}_0(\Omega).$$
Since $w$ is a supersolution, by Definition \ref{d: sol} we have $f(w)\in L^2(\Omega)$. Therefore thanks to \eqref{assum-a}, we get $\nu\in L^2(\Omega)$. Moreover, $w$ is a supersolution implies $\nu\geq 0$.  $w$ is not a solution implies $\nu\not\equiv0$. Consider the following problem:
\begin{equation*}
\begin{cases}
\Delta^2\psi=\nu \quad\text{in}\quad\Omega,\\
\psi=0=\Delta\psi \quad\text{on}\quad\bdw.
\end{cases}
\end{equation*}
We can break this problem into system of second order Dirichlet problem by defining $$-\Delta\psi=\tilde{\psi} \quad\text{in}\quad\Omega, \quad \psi=0 \quad\text{on}\quad\bdw, $$  
$$-\Delta\tilde{\psi}=\nu \quad\text{in}\quad\Omega, \quad \tilde{\psi}=0 \quad\text{on}\quad\bdw.$$ Then by weak maximum principle it is easy to check that 
$\psi>\epsilon\delta(x)$ for some $\epsilon>0$, where $\delta(x)=\text{dist}(x,\bdw).$ Next we consider the problem:
\begin{equation*}
\begin{cases}
\Delta^2\eta=b \quad\text{in}\quad\Omega,\\
\eta=0=\Delta\eta \quad\text{on}\quad\bdw.
\end{cases}
\end{equation*}
As before we break this problem into system of equations as follows:
$$-\Delta\eta=\tilde{\eta} \quad\text{in}\quad\Omega, \quad \eta=0 \quad\text{on}\quad\bdw, $$  
$$-\Delta\tilde{\eta}=b \quad\text{in}\quad\Omega, \quad \tilde{\eta}=0 \quad\text{on}\quad\bdw.$$ 
Since $b\in L^p(\Omega)$ for some $p>\frac{N}{3}$, using theory of elliptic regularity and Soblev embedding theorem, we get $\tilde{\eta}\in L^{p^*}(\Omega)$ where $p^*=\frac{Np}{N-2p}>N$. Therefore $\eta\in C^{1,\alpha}(\Omega)$ for some $\alpha\in(0,1)$. Hence $\eta<C\delta(x)$ in $\Omega$ for some $C\in(0,\infty)$.  Define, $v=w+\epsilon C^{-1}\eta-\psi$. Clearly $v<w$ in $\Omega$ and $v\in W^{2,2}(\Omega)\cap W^{1,2}_0(\Omega)$. Also,
$$\Delta^2v=\Delta^2w+\epsilon C^{-1}b-\nu=\mu a(x)w+f(w)+\la^{*}b+\nu+\epsilon C^{-1}b-\nu\geq \mu a(x)v+f(v)+(\la^{*}+\epsilon C^{-1})b.$$
As a result, $v$ is a supersolution to $(P_{\la^{*}+\epsilon C^{-1}})$. Hence $(P_{\la^{*}+\epsilon C^{-1}})$ has a solution contradicting the extremality of $\la^{*}$. 
\hfill$\square$
\end{proof}
The next lemma is in the spirit of \cite[Lemma 3.2]{DN}.
\begin{Lemma}
If $u\in L^1(\Omega)$ is an nonnegative distributional solution of $\Delta^2 u=h \quad\text{in}\quad\Omega$, where $h\in L^1(\Omega)$, then ess inf $u=0$. 
\end{Lemma}
\begin{proof}
Assume the lemma does not hold, that is there exists $\epsilon>0$ such that $u\geq \epsilon>0$ a.e. in $\Omega$. We extend $u$ and $h$ by $0$ in $\Rn\setminus \Omega$. Let $\rho_n$ denote the standard molifier. Define $u_n=u\star\rho_n$ and $h_n=h\star\rho_n$. Following the same argument as in \cite[Lemma 3.2]{DN}, we can show that, there exists $\alpha>0$ such that for $n$ large enough $u_n\geq \alpha\epsilon$ everywhere in $\Omega$ and given $\omega\Subset\Omega$ and $n$ large enough, $\Delta^2 u_n=h_n$ everywhere in $\omega$. Let $\phi$ solve the following:
\begin{equation}\label{app-phi}
\begin{cases}
\Delta^2\phi=1 \quad\text{in}\quad\omega,\\
\phi=0=\Delta\phi \quad\text{on}\quad\pa\omega.
\end{cases}
\end{equation}
Integrating by parts we obtain
$$\int_{\omega}u_n dx=\int_{\omega}u_n\Delta^2\phi dx=\int_{\omega}\Delta u_n\Delta\phi dx+\int_{\pa\omega}\frac{\pa}{\pa n}(\Delta\phi)u_n ds=\int_{\omega}h_n\phi dx+\int_{\pa\omega}\frac{\pa}{\pa n}(\Delta\phi)u_n ds.$$
Thus,
$$\int_{\omega}h_n\phi dx-\int_{\omega}u_n dx=-\int_{\pa\omega}\frac{\pa}{\pa n}(\Delta\phi)u_n ds\leq -\alpha\epsilon|\omega|,$$
since $\displaystyle\int_{\pa\omega}\frac{\pa}{\pa n}(\Delta\phi)ds=|\omega|$ (follows from \eqref{app-phi} after integrating by parts).
Since $u_n\to u$ in $L^1(\Omega)$, $h_n\to h$ in $L^1(\Omega)$ we get
$$\int_{\omega}h\phi dx-\int_{\omega}u dx\leq -\alpha\epsilon|\omega|.$$
Next choose $\omega=\omega_n:=\{x\in\Omega: \text{dist}(x,\bdw)>\frac{1}{n}\}$, $n\to\infty$. Let $\phi_n$ denote the corresponding solution to \eqref{app-phi} in $\omega_n$. Then $\phi_n\uparrow\phi$ where $\phi$ solves
\begin{equation*}
\begin{cases}
\Delta^2\phi=1 \quad\text{in}\quad\Omega,\\
\phi=0=\Delta\phi \quad\text{on}\quad\pa\Omega.
\end{cases}
\end{equation*}
Taking limit $n\to\infty$ in $\int_{\omega_n}h\phi_n dx-\int_{\omega_n}u dx\leq -\alpha\epsilon|\omega_n|$ and using $\Delta^2 u=h \quad\text{in}\quad\Omega$, we have $0\leq-\alpha\epsilon|\Omega|$. This gives a contradiction. 
\hfill$\square$
\end{proof}
\begin{Theorem}
Assume \eqref{assum-mu} is satisfied.  Then problem \eqref{a-n} has a nonnegative minimal solution for every $\la>0$.
\end{Theorem}
\begin{proof}
{\bf Step 1:} Assume $a\in L^1_{loc}(\Omega)$ which satisfies \eqref{assum-a}.  Let $b\in L^{\infty}(\Omega)$ and $f\in L^{\infty}(\R^+)\cap C(\R^+) $ be nonnegative functions, $b\not\equiv 0$ and $\la>0$. Then there exists $u\in W^{2,2}(\Omega)\cap W^{1,2}_0(\Omega)$ such that $u$ solves $(P_{\la})$ for all $\la>0$.

To prove step 1, let $u_0\in W^{2,2}(\Omega)\cap W^{1,2}_0(\Omega)$ be a positive solution to 
\begin{equation*}
\begin{cases}
\Delta^2 u_0=\la b \quad\text{in}\quad\Omega,\\
u_0=0=\Delta u_0 \quad\text{on}\quad\bdw.
\end{cases}
\end{equation*} 
Since $\la b\in L^{\infty}(\Omega)\subset L^2(\Omega)$ we get $u_0\in W^{2,2}(\Omega)\cap W^{1,2}_0(\Omega)$. Next, using iteration we will show that there exists $u_n\in W^{2,2}(\Omega)\cap W^{1,2}_0(\Omega)$ for $n=1,2, \cdots$ such that $u_n$ solves the following problem:
\begin{equation}\label{app:un-1}
\begin{cases}
\Delta^2 u_n=\mu a(x)u_{n-1}+f(u_{n-1})+\la b(x) \quad\text{in}\quad\Omega,\\
u_n=0=\Delta u_n \quad\text{on}\quad\bdw.
\end{cases}
\end{equation} 
Thanks to \eqref{assum-a} and the assumptions that $f, b\in L^{\infty}(\Omega)$, it follows that $\mu a(x)u_0+f(u_0)+\la b(x)\in L^2(\Omega)$. Therefore $u_1$ is well defined. Moreover, by comparison principle $0<u_0\leq u_1$. Using the induction method, similarly we can show $u_n$ is well defined and $0<u_0\leq u_1\leq\cdots\leq u_n\leq\cdots$ . 

{\bf Claim:} $\{u_n\}$ is uniformly bounded in $W^{2,2}(\Omega)\cap W^{1,2}_0(\Omega)$.\\
To see this, lets note that from \eqref{app:un-1} we can write
\begin{equation}\label{app-un2}
|\Delta u_n|^2_{L^2(\Omega)} =\Iom\displaystyle\left(\mu a(x)u_{n-1}+f(u_{n-1})+\la b(x)\right)u_n dx .
\end{equation}
Using Holder inequality, \eqref{assum-a} and  Young's inequality, the terms on the RHS can be simplified as follows
$$\la\Iom bu_n dx\leq\la|b|_{L^{\infty}(\Omega)}|\Omega|^\frac{1}{2}|u_n|_{L^2(\Omega)}\leq\frac{C}{\sqrt{\ga}}|b|_{L^{\infty}(\Omega)}|\Delta u_n|_{L^2(\Omega)}\leq \epsilon|\Delta u_n|^2_{L^2(\Omega)}+c(\epsilon)|b|^2_{L^{\infty}(\Omega)}.$$
\begin{eqnarray*}
\Iom f(u_{n-1})u_n dx \leq |f|_{L^{\infty}(\Omega)}|\Omega|^\frac{1}{2}|u_n|_{L^2(\Omega)} &\leq& \frac{C}{\sqrt{\ga}}|f|_{L^{\infty}(\Omega)}|\Delta u_n|_{L^2(\Omega)}\\
&\leq& \epsilon|\Delta u_n|^2_{L^2(\Omega)}+c(\epsilon)|f|^2_{L^{\infty}(\Omega)}. 
\end{eqnarray*}
$$\mu\Iom a(x)u_{n-1}u_n dx\leq\mu\Iom a(x)u_n^2dx\leq\mu |a(x)u_n|_{L^2(\Omega)}|u_n|_{L^2(\Omega)}\leq \frac{\mu}{\sqrt{\ga}}|\Delta u_n|^2_{L^2(\Omega)}. $$
Since $\frac{\mu}{\sqrt{\ga}}<1$, we can choose $\epsilon>0$ such that $2\epsilon+\frac{\mu}{\sqrt{\ga}}<1$. Substituting this $\epsilon$ in above three inequalities and combining them with \eqref{app-un2}, we have
$$|\Delta u_n|^2_{L^2(\Omega)}\leq C(|b|_{L^{\infty}(\Omega)}+|f|_{L^{\infty}(\Omega)}).$$ 
This proves the claim. As a consequence there exists  $u\in W^{2,2}(\Omega)\cap W^{1,2}_0(\Omega)$ such that upto a subsequence $u_n\deb u$ in $W^{2,2}(\Omega)\cap W^{1,2}_0(\Omega)$ and $u_n\to u$ in $L^2(\Omega)$. Therefore we can conclude the theorem as we did in Lemma \ref{l:supersol}.

{\bf Step 2:} Let  $\{b_n(x)\}$ and $\{f_n\}$ be increasing sequence of bounded functions converging pointwise respectively to  $b(x)$ and $f$ ($f_n$ is continuous for $n=1,2, \cdots$ ). Then by Step 1, there exists a nonnegative minimal solution $v_n\in W^{2,2}(\Omega)\cap W^{1,2}_0(\Omega)$ of the following problem:
\begin{equation}\label{app:a-n3}
\begin{cases}
\Delta^2 v_n-\mu a(x)v_n=f_n(v_n)+\la b_n(x) \quad\text{in}\quad\Omega,\\
v_n=0=\Delta v_n \quad\text{on}\quad\bdw.
\end{cases}
\end{equation}
Clearly $v_n$ is a nonnegative supersolution to \eqref{a-n}. Therefore the theorem follows from Lemma \ref{l:supersol}.

\hfill$\square$
\end{proof}

{\bf Final Remark:}  The results of this paper can be easily extended to the equations of the form  $$\Delta^2 u-\mu a(x)u=c(x)f(u)+\la b(x) \quad\text{in}\quad\Omega,$$ where $c\in L^1_{loc}(\Omega)$ is a nonnegative function. In particular, \eqref{ex-cond} will be changed to
$$c f(\epsilon\zeta_1)\in L^2(\Omega) \quad\text{and}\quad G(c(x)f(\epsilon\zeta_1))\leq C\zeta_1.$$

\vspace{2mm}

{\bf Acknowledgement}: The author would like to gratefully thank Dr. Anup Biswas for many fruitful discussions that led to various results in this work. 
This work is supported by INSPIRE research grant DST/INSPIRE 04/2013/000152.

\label{References}

\end{document}